\documentclass[11pt]{article}%
\usepackage{amsmath}
\usepackage{amsfonts}
\usepackage{amssymb}
\usepackage{graphicx}
\usepackage{hyperref}%
\newtheorem{theorem}{Theorem}[section]

\newtheorem{corollary}[theorem]{Corollary}

\newtheorem{example}[theorem]{Example}

\newtheorem{lemma}[theorem]{Lemma}

\newtheorem{proposition}[theorem]{Proposition}

\newenvironment{proof}[1][Proof]{\noindent\textbf{#1.} }{\ \rule{0.5em}{0.5em}}
\begin{document}

\title{Unitization of a lattice ordered ring with a truncation}
\author{Karim Boulabiar\thanks{Corresponding author}\quad and\quad Mounir
Mahfoudhi\bigskip\\Universit\'{e} de Tunis El Manar, Facult\'{e} des Sciences de Tunis\\LR11ES12 Alg\`{e}bre, Topologie, Arithm\'{e}tique et Ordre\\2092, Tunis, Tunisie}
\date{}
\maketitle

\begin{abstract}
Let $R$ be a lattice ordered ring along with a truncation in the sense of
Ball. We give a necessary and sufficient condition on $R$ for its unitization
$R\oplus\mathbb{Q}$ to be again a lattice ordered ring. Also, we shall see
that $R\oplus\mathbb{Q}$ is a lattice ordered ring for at most one truncation.
Particular attention will be paid to the Archimedean case. More precisely, we
shall identify the unique truncation on an Archimedean $\ell$-ring $R$ which
makes $R\oplus\mathbb{Q}$ into a lattice ordered ring.

\end{abstract}



\section{Introduction}

Where reference is made to an $\ell$-group, it shall always mean a divisible
abelian lattice-ordered group (i.e., a vector lattice over the rationals
$\mathbb{Q}$). Recently in his pioneer paper \cite{B2014}, Ball defined a
\textsl{truncation }on an $\ell$-group $G$ with positive cone $G^{+}$ to be a
function $\tau:G^{+}\rightarrow G^{+}$ with the following properties.

\begin{enumerate}
\item[$\left(  \tau_{1}\right)  $] $x\wedge\tau\left(  y\right)  \leq
\tau\left(  x\right)  \leq x$ for all $x,y\in G^{+}$.

\item[$\left(  \tau_{2}\right)  $] If $x\in G^{+}$ and $\tau\left(  x\right)
=0$, then $x=0$.

\item[$\left(  \tau_{3}\right)  $] If $x\in G^{+}$ and $nx=\tau\left(
nx\right)  $ for all $n\in\left\{  1,2,...\right\}  $, then $x=0$.
\end{enumerate}

\noindent For an $\ell$-group $G$ with a truncation $\tau$, we put%
\[
\tau\left(  G^{+}\right)  =\left\{  \tau\left(  x\right)  :x\in G^{+}\right\}
.
\]
It is shown in \cite{BE2017} that the cardinal sum $G\oplus\mathbb{Q}$ can be
endowed with an ordering such that $G\oplus\mathbb{Q}$ is an $\ell$-group with
$1$ as a weak order unit and $G$ as an $\ell$-ideal (i.e., a convex $\ell
$-subgroup). The positive cone of $G\oplus\mathbb{Q}$ is the set%
\[
G^{+}\cup\left\{  x+p:p>0\text{ and }\frac{1}{p}x^{-}\in\tau\left(
G^{+}\right)  \right\}
\]
(where $x^{-}$ denotes the negative part of $x$). In \cite{BHM2018}, the
$\ell$-group $G\oplus\mathbb{Q}$ is called the \textsl{Alexandroff
unitization} of $G$. For brevity, the Alexandroff unitization of $G$ resulting
from the truncation $\tau$ will be denoted by $\tau G$.

The present paper is developed around the following problem (a look, from a
different angle, at a similar problem can be found in \cite{HJ2007} by Hager
and Johnson). Let $R$ be an $\ell$-ring with a truncation $\tau$. Drawing
plenty of inspiration from the the classical unitization process in Banach
Algebra Theory (see, e.g., \cite{BD1973}), a natural multiplication can be
introduced on the $\ell$-group $\tau R=R\oplus\mathbb{Q}$ by putting%
\begin{equation}
\left(  x+p\right)  \left(  y+q\right)  =xy+qx+py+pq\text{,\quad for all
}x,y\in R\text{ and }p,q\in\mathbb{Q}. \tag{$\ast$}\label{E1}%
\end{equation}
It is routine to check that this multiplication makes $\tau R$ into an
associative ring with $1$ as identity and $R$ as a ring ideal. It would seem
plausible to think that $\tau R$ is even an $\ell$-ring, i.e., the positive
cone of $\tau R$ is closed under the multiplication defined in (\ref{E1}).
Nevertheless, as the next example shows, such an attractive result cannot be
expected without imposing an extra compatibility condition.

\begin{example}
\label{Exp}Let $C\left(  \mathbb{R}\right)  $ denote the $\ell$-group of all
real-valued continuous functions on the real line $\mathbb{R}$. Clearly, the
function $\tau:C\left(  \mathbb{R}\right)  ^{+}\rightarrow C\left(
\mathbb{R}\right)  ^{+}$ defined by%
\[
\tau\left(  x\right)  \left(  r\right)  =\min\left\{  x\left(  r\right)
,1\right\}  \text{,\quad for all }x\in C\left(  \mathbb{R}\right)  \text{ and
}r\in\mathbb{R}%
\]
is a truncation on $C\left(  \mathbb{R}\right)  $. Moreover, it is easily
checked that $C\left(  \mathbb{R}\right)  $ is an $\ell$-ring under the
multiplication given by%
\[
\left(  xy\right)  \left(  r\right)  =2x\left(  r\right)  y\left(  r\right)
\text{,\quad for all }x,y\in C\left(  \mathbb{R}\right)  \text{ and }%
r\in\mathbb{R}.
\]
Define $x,y\in C\left(  \mathbb{R}\right)  $ by%
\[
x\left(  r\right)  =\cos r\text{\quad and\quad}y\left(  r\right)  =\sin
r\text{,\quad for all }r\in\mathbb{R}.
\]
Obviously, we have $x^{-},y^{-}\in\tau\left(  C\left(  \mathbb{R}\right)
^{+}\right)  $ and so $x+1,y+1\geq0$ in $\tau C\left(  \mathbb{R}\right)  $.
Furthermore, a simple calculation leads to the equalities%
\[
\left(  x+1\right)  \left(  y+1\right)  =xy+x+y+1\text{ and }\left(
xy+x+y\right)  ^{-}\left(  -\pi/4\right)  =2.
\]
It follows that%
\[
\left(  xy+x+y\right)  ^{-}\notin\tau\left(  C\left(  \mathbb{R}\right)
^{+}\right)  .
\]
This shows that $\left(  x+1\right)  \left(  y+1\right)  $ is not positive in
$\tau C\left(  \mathbb{R}\right)  $, which reveals that the multiplication
defined in \emph{(\ref{E1})} does not make $\tau C\left(  \mathbb{R}\right)  $
into an $\ell$-ring.
\end{example}

Thus, we have to reword the question as follows. What does an $\ell$-ring $R$
with a truncation still lack in order to produce a satisfactory outcome? The
main purpose of this paper is to look for necessary and sufficient conditions
on $R$ for $\tau R$ to be an $\ell$-ring. In this regard, we shall prove,
among other facts, that $\tau R$ is an $\ell$-ring if and only if $R$ is an
almost $f$-ring such that%
\[
\tau\left(  R^{+}\right)  =\left\{  x\in R:x^{2}\leq x\right\}  .
\]
To prove this equivalence, we shall show that if $\tau R$ is an $\ell$-ring,
then it is automatically a reduced $f$-ring. Also, we shall prove, as a
consequence of the main result, that any $\ell$-ring has at most one
truncation $\tau$ such that $\tau R$ is an $f$-ring.

Particular attention will be paid to the Archimedean case. First, recall that
any reduced Archimedean $f$-ring $R$ can be embedded as an $\ell$-subring in
the unital Archimedean $f$-ring $\mathrm{Orth}\left(  R\right)  $ of all
orthomorphisms on $R$ (see, e.g., \cite[Theorem 12.3.8]{BKW1977}). Following
\cite{BBE2014} (see also \cite{HP1986,HP1984}), we call the reduced
Archimedean $f$-ring $R$ a \textsl{Stone }$f$\textsl{-ring} if%
\[
\mathrm{id}_{R}\wedge x\in R,\text{\quad for all }x\in R^{+},
\]
where $\mathrm{id}_{R}$ denotes the identity map on $R$ (which is the identity
of the ring $\mathrm{Orth}\left(  R\right)  $). As an application of our
aforementioned results, we shall prove that if $R$ is an Archimedean $\ell
$-ring with a truncation $\tau$, then $\tau R$ is an $\ell$-ring if and only
if $R$ is a Stone $f$-ring and $\tau$ is given by%
\[
\tau\left(  x\right)  =\mathrm{id}_{R}\wedge x,\text{\quad for all }x\in
R^{+}.
\]
Finally, we refer the reader to the classical monographs \cite{BKW1977,D1995}
for unexplained terminology and notation.

\section{Preliminaries on $\ell$ -groups with truncation}

This short section is a summation of the recent paper \cite{BE2017} on $\ell$
-groups with truncation. More precisely, we shall collect results from this
reference that are relevant to our present work.

First, it could be useful to emphasize that any $\ell$-group under
consideration is assumed to be divisible and abelian. Moreover, we recall for
convenience that a \textsl{truncation} on an $\ell$-group $G$ is a function
$\tau:G^{+}\rightarrow G^{+}$ such that

\begin{enumerate}
\item[$\left(  \tau_{1}\right)  $] $x\wedge\tau\left(  y\right)  \leq
\tau\left(  x\right)  \leq x$ for all $x,y\in G^{+}$.

\item[$\left(  \tau_{2}\right)  $] If $x\in G^{+}$ and $\tau\left(  x\right)
=0$ then $x=0$.

\item[$\left(  \tau_{3}\right)  $] If $x\in G^{+}$ and $nx=\tau\left(
nx\right)  $ for all $n\in\left\{  1,2,...\right\}  $ then $x=0$.
\end{enumerate}

\noindent Throughout the paper, the range of the truncation $\tau$ is denoted
by $\tau\left(  G^{+}\right)  $, that is,%
\[
\tau\left(  G^{+}\right)  =\left\{  \tau\left(  x\right)  :x\in G^{+}\right\}
.
\]
One may prove that

\begin{enumerate}
\item[$\left(  \upsilon_{1}\right)  $] $\tau\left(  G^{+}\right)  =\left\{
x\in G^{+}:\exists y\in G^{+},x\leq\tau\left(  y\right)  \right\}  =\left\{
x\in G^{+}:\tau\left(  x\right)  =x\right\}  $.
\end{enumerate}

\noindent A positive element $e$ in an $\ell$-group $G$ with a truncation
$\tau$ is called a\textsl{ unit} \textsl{for} $\tau$ if $\tau$ is given by
meet with $e$, i.e.,%
\[
\tau\left(  x\right)  =e\wedge x,\quad\text{for all }x\in G^{+}.
\]
In this situation, $e$ is automatically a weak unit in $G$ (this follows
straightforwardly from $\left(  \tau_{2}\right)  $). It is worth noting that,
in general, an $\ell$-group $G$ with a truncation $\tau$ does not contain a
unit for $\tau$ (see Example 1.3 in \cite{BHM2018}).

At this point, we focus on the unitization of an $\ell$-group $G$ with a
truncation $\tau$. The Cartesian product $G\times\mathbb{Q}$ is a divisible
group with respect to the coordinatewise addition. Moreover, it is not hard to
see that $G$ can be identified with the subgroup $G\times\left\{  0\right\}  $
of $G\times\mathbb{Q}$. Therefore, we may assume that $G$ is a subgroup of
$G\times\mathbb{Q}$. Similarly, we may identify $\mathbb{Q}$ with $\left\{
0\right\}  \times\mathbb{Q}$ and so $\mathbb{Q}$ can also be seen as a
subgroup of $G\times\mathbb{Q}$. Actually, we have the cardinal sum%
\[
G\times\mathbb{Q}=G\oplus\mathbb{Q}=\left\{  x+p:x\in G\text{ and }%
p\in\mathbb{Q}\right\}  .
\]
Accordingly, each element $v$ in $G\oplus\mathbb{Q}$ can be uniquely written
as a sum of an element of $G$ and a rational number. It turns out that
$G\oplus\mathbb{Q}$ can be endowed with an ordering such that the following
properties hold.

\begin{enumerate}
\item[$\left(  \upsilon_{2}\right)  $] $G\oplus\mathbb{Q}$ is an $\ell$-group
with positive cone the set%
\[
G^{+}\cup\left\{  x+p\in G\oplus\mathbb{Q}:p>0\text{ and }\frac{1}{p}x^{-}%
\in\tau\left(  G^{+}\right)  \right\}
\]

\item[$\left(  \upsilon_{3}\right)  $] The positive part of an element $x+p\in
G\oplus\mathbb{Q}$ is given by%
\[
\left(  x+p\right)  ^{+}=\left\{
\begin{array}
[c]{l}%
x^{+}-p\tau\left(  \dfrac{1}{p}x^{-}\right)  +p,\text{\quad if }p>0\\
\\
x^{+}+p\tau\left(  \dfrac{-1}{p}x^{+}\right)  ,\text{\quad if }p<0.
\end{array}
\right.
\]

\item[$\left(  \upsilon_{4}\right)  $] $G\oplus\mathbb{Q}$ has $1$ as a
positive weak unit (i.e., if $v\in G\oplus\mathbb{Q}$ with $\left\vert
v\right\vert \wedge1=0$, then $v=0$).

\item[$\left(  \upsilon_{5}\right)  $] $\tau\left(  x\right)  =1\wedge x$ for
all $x\in G^{+}$.

\item[$\left(  \upsilon_{6}\right)  $] $G$ is an $\ell$-ideal (i.e., a convex
$\ell$-subgroup) in $G\oplus\mathbb{Q}$.
\end{enumerate}

\noindent For proofs of $\left(  \upsilon_{2}\right)  -\left(  \upsilon
_{6}\right)  $, the reader can consult \cite{BE2017}. Moreover, from $\left(
\upsilon_{2}\right)  $ it follows quickly that

\begin{enumerate}
\item[$\left(  \upsilon_{7}\right)  $] if $x\in G$ and $p\in\mathbb{Q}$ such
that $x+p\geq0$ in $G\oplus\mathbb{Q}$, then $p\geq0$.
\end{enumerate}

\noindent We end this section by noting that the $\ell$-group $G\oplus
\mathbb{Q}$ is called in \cite{BHM2018} the \textsl{Alexandroff unitization}
of $G$. We shall denote it by $\tau G$ throughout the paper because we will
often need to point out that it results from the truncation $\tau$.

\section{Results on $\ell$-rings}

In this section, we collect some fundamental properties on $\ell$-rings we
will need in what follows. All given rings are assumed to be divisible with
associative multiplication, but multiplication need not be commutative and
there need not be a multiplicative identity element. So, rings we are dealing
with are associative algebras (in the sense \cite{BD1973}) over the rationals
$\mathbb{Q}$. A ring $R$ is called a \textsl{lattice-ordered ring} (an $\ell
$-\textsl{ring} in short) if its underlying group is an $\ell$-group such that
the positive cone $R^{+}$ is closed under multiplication, i.e.,%
\[
xy\in R^{+}\text{,}\quad\text{for all }x,y\in R^{+}.
\]
The $\ell$-ring $R$ is said to be \textsl{unital} if it has an identity $e$,
i.e., $ex=xe=x$. The $\ell$-ring $R$ is said to be \textsl{reduced} (or,
\textsl{semiprime}) if $R$ contains no nonzero nilpotent elements. It is an
easy exercise to show that the $\ell$-ring $R$ is reduced if and only if
$x^{2}=0$ in $R$ implies $x=0$. We call the $\ell$-ring $R$ an \textsl{almost
}$f$\textsl{-ring} after Birkhoff in \cite{B1967} if%
\[
xy=0\text{,}\quad\text{for all }x,y\in R\text{ with }x\wedge y=0.
\]
Actually, the $\ell$-ring $R$ is an almost $f$-ring if and only if%
\[
x^{+}x^{-}=0\text{,}\quad\text{for all }x\in R
\]
(see Proposition 1..3 in \cite{BH1990}). Hence, if $R$ is an almost $f$-ring
then%
\[
x^{2}=\left\vert x\right\vert ^{2}=\left\vert x^{2}\right\vert \text{,}%
\quad\text{for all }x\in R.
\]
This means in particular that any almost $f$-ring has positive squares. For
instance, if the almost $f$-ring $R$ has an identity $e$ then $e\in R^{+}$.
The following proposition will play a key role in the proof of the main result
of this paper. For the proof, see Theorem 1.9 in \cite{BH1990} or Theorem 15
in \cite{BP1956}. First, we have to notice that by a \textsl{weak unit} in the
$\ell$-group $G$ we shall mean an element $e\in G$ such that $0<e$ and, for
every $x\in G$, if $\left\vert x\right\vert \wedge e=0$, then $x=0$.

\begin{proposition}
\label{BP}Any $\ell$-ring with an identity $e>0$ is an almost $f$-ring if and
only if $e$ is a weak unit.
\end{proposition}

Birkhoff and Pierce in \cite{BP1956} call the $\ell$-ring $R$ an
$f$-\textsl{ring} if%
\[
x\wedge y=0\text{ and }z\geq0\quad\text{imply\quad}\left(  xz\right)  \wedge
y=\left(  zx\right)  \wedge y=0.
\]
Obviously, any $f$-ring is an almost $f$-ring, but not conversely (see Example
on page 62 in \cite{BP1956}). It follows that $f$-rings enjoy all properties
of almost $f$-rings. The proof of the following result can be found in
\cite[Theorem 1.11]{BH1990}.

\begin{proposition}
\label{BH} A reduced $\ell$-ring is an $f$-ring if and only if it an almost
$f$-ring.
\end{proposition}

The reader would realize that the manuscript \cite{BH1990} by Bernau and
Huijsmans will be of great use in this paper. It should be pointed out that
this reference deals with algebras rather than rings. But, as its authors
themselves observed (see Page 1 in \cite{BH1990}), all results and proofs
still work for rings.

We call after Ball in \cite{B2014} an element $x$ in an $\ell$-ring $R$ with
identity $e>0$ an \textsl{infinitesimal} if%
\[
n\left\vert x\right\vert \leq e\text{,}\quad\text{for all }n\in\left\{
1,2,...\right\}  .
\]
The next lemma will be useful for later purposes.

\begin{lemma}
\label{diem}Let $R$ be a unital $\ell$-ring such that its identity is
simultaneously a weak unit. If $R$ has no non-zero infinitesimals, then $R$ is
a reduced $f$-ring.
\end{lemma}

\begin{proof}
Let $e$ denote the identity of $R$. By Proposition \ref{BP}, the $\ell$-ring
$R$ is an almost $f$-ring and so has positive squares. Let $n\in\left\{
1,2,...\right\}  $ and $x\in R$ such that $x^{2}=0$. Hence, $n^{2}x^{2}=0$ and
so%
\[
n\left\vert x\right\vert \leq2n\left\vert x\right\vert =2n\left\vert
x\right\vert -n^{2}x^{2}=e-\left(  n\left\vert x\right\vert -e\right)
^{2}\leq e.
\]
We derive that $x$ is an infinitesimal in $R$ and thus $x=0$. This shows that
$R$ is reduced. In summary, $R$ is a reduced almost $f$-ring. This together
with Proposition \ref{BH} ends the proof of the lemma.
\end{proof}

The last result of this section is presumably known, though we have not been
able to locate a reference for it.

\begin{lemma}
\label{semi}Let $R$ be a reduced $f$-ring and $0\leq x\in R$. Then the
following are equivalent.

\begin{enumerate}
\item[\emph{(i)}] $x^{2}\leq x$.

\item[\emph{(ii)}] $xy\leq y$ for all $y\in R^{+}$.

\item[\emph{(iii)}] $yx\leq y$ for all $y\in R^{+}$.
\end{enumerate}
\end{lemma}

\begin{proof}
We start the proof with a preliminary observation. Let $y,z\in R^{+}$ such
that $yz=0$. Then,%
\[
\left(  zy\right)  ^{2}=zyzy=z\left(  yz\right)  y=0.
\]
Since $R$ is reduced, we get $zy=0$. Now, the implication
$\mathrm{(ii)\Rightarrow(i)}$ is trivial (and so is the implication
$\mathrm{(iii)\Rightarrow(i)}$). We only prove the implication
$\mathrm{(i)\Rightarrow(ii)}$ (the implication $\mathrm{(i)\Rightarrow(iii)}$
can be obtained in the same way).

Assume that $x^{2}\leq x$ and pick $y\in R^{+}$. Clearly, $x^{2}y-xy\leq0$ and
so $x\left(  xy-y\right)  \leq0$. Hence, using%
\[
\left(  xy-y\right)  ^{-}\left(  xy-y\right)  ^{+}=0,
\]
we get%
\[
0\leq x\left[  \left(  xy-y\right)  ^{+}\right]  ^{2}=x\left(  xy-y\right)
\left(  xy-y\right)  ^{+}\leq0.
\]
Thus,%
\[
x\left[  \left(  xy-y\right)  ^{+}\right]  ^{2}=0.
\]
In view of the preliminary observation, we derive that%
\[
\left[  \left(  xy-y\right)  ^{+}\right]  ^{2}x=0.
\]
Therefore,%
\[
0\leq\left[  \left(  xy-y\right)  ^{+}\right]  ^{3}=\left[  \left(
xy-y\right)  ^{+}\right]  ^{2}\left(  xy-y\right)  ^{+}\leq\left[  \left(
xy-y\right)  ^{+}\right]  ^{2}xy=0.
\]
Since $R$ is reduced, we obtain $\left(  xy-y\right)  ^{+}=0$. It follows that
$xy\leq y$, completing the proof of the lemma.
\end{proof}

\section{The Alexandroff unitization of an $\ell$-ring with a truncation}

The central purpose of this section is to give a complete answer to the
following question. When is the Alexandroff unitization $\tau R$ of an $\ell
$-ring $R$ with a truncation $\tau$ an $\ell$-ring? Our investigation starts
with the following proposition.

\begin{proposition}
\label{reduced}Let $R$ be an $\ell$-ring with a truncation $\tau$. Then $\tau
R$ is an $\ell$-ring if and only if $\tau R\ $is a reduced $f$-ring.
\end{proposition}

\begin{proof}
Only Necessity needs a proof. Assume that $\tau R$ is an $\ell$-ring. Since
$1$ is an identity and, simultaneously, a weak unit in $\tau R$, we conclude
that $\tau R$ is an almost $f$-ring (where we use Proposition \ref{BP}). In
particular, squares in $\tau R$ are positive. We claim that $R$ has no
non-trivial infinitesimals. To this end, let $x\in R$ and $p\in\mathbb{Q}$
such that%
\[
n\left\vert x+p\right\vert \leq1,\quad\text{for all }n\in\left\{
1,2,...\right\}  .
\]
From $\left(  \upsilon_{3}\right)  $ it follows quite easily that that, if
$p\neq0$, then%
\[
\left\vert x+p\right\vert =\left\vert x\right\vert -2\left\vert p\right\vert
\tau\left(  \frac{1}{p}x^{-}\wedge\frac{-1}{p}x^{+}\right)  +\left\vert
p\right\vert .
\]
Thus, $\left\vert x+p\right\vert $ is always of the form $y+\left\vert
p\right\vert $ for some $y\in R$. We derive that%
\[
0\leq-ny+\left(  1-n\left\vert p\right\vert \right)  ,\quad\text{for all }%
n\in\left\{  1,2,...\right\}  .
\]
But then%
\[
n\left\vert p\right\vert \leq1,\quad\text{for all }n\in\left\{
1,2,...\right\}
\]
(where we use $\left(  \upsilon_{7}\right)  $) and thus $p=0$. Whence,%
\[
n\left\vert x\right\vert \leq1,\quad\text{for all }n\in\left\{
1,2,...\right\}  .
\]
Consequently, $\left(  \upsilon_{5}\right)  $ leads to%
\[
n\left\vert x\right\vert =\tau\left(  n\left\vert x\right\vert \right)
,\quad\text{for all }n\in\left\{  1,2,...\right\}  .
\]
This together with $\left(  \tau_{3}\right)  $ yields that $x=0$. Accordingly,
$0$ is the only infinitesimal in $\tau R$. Using Lemma \ref{diem}, we infer
that $\tau R$ is reduced. In summary, $\tau R$ is a reduced almost $f$-ring.
In view of Proposition \ref{BH}, we derive that $\tau R$ is a reduced
$f$-ring, which ends the proof of the proposition.
\end{proof}

We are in position at this point to state and prove the central result of this work.

\begin{theorem}
\label{main}Let $R$ be an $\ell$-ring with a truncation $\tau$. Then $\tau R$
is an $\ell$-ring if and only if $R$ is a reduced $f$-ring with%
\[
\tau\left(  R^{+}\right)  =\left\{  x\in R:x^{2}\leq x\right\}  .
\]

\end{theorem}

\begin{proof}
\textit{Necessity.} If $\tau R$ is an $\ell$-ring, then $\mathfrak{\tau}R$ is
a reduced $f$-ring (by Proposition \ref{reduced}). Thus, $R$ is a reduced
$f$-ring since $R$ is a subring and a sublattice of $\tau R$. Now, we prove
the equality%
\[
\tau\left(  R^{+}\right)  =\left\{  x\in R:x^{2}\leq x\right\}  .
\]
If $x\in\tau\left(  R^{+}\right)  $ then $0\leq x\leq1$ in $\tau R$ and so
$x^{2}\leq x$. Conversely, suppose that $x^{2}\leq x$. This inequality holds
in $\tau R$ which is an almost $f$-ring. It follows in particular that
$x\geq0$ because squares in an almost $f$-ring are positive. Hence, we can
apply Lemma \ref{semi} to write%
\[
xy\leq y\text{,\quad for all }0\leq y\in\tau R.
\]
We directly get $x\leq1$ and so $x\in\tau\left(  R^{+}\right)  $. This ends
the proof of Necessity.

\textit{Sufficiency}. Suppose that $R$ is a reduced $f$-ring such that%
\[
\tau\left(  R^{+}\right)  =\left\{  x\in R:x^{2}\leq x\right\}  .
\]
By Lemma \ref{semi}, we have%
\begin{equation}
\tau\left(  R^{+}\right)  =\left\{  x\in R^{+}:xy\leq y\text{ for all }y\in
R^{+}\right\}  \label{1}%
\end{equation}
and%
\begin{equation}
\tau\left(  R^{+}\right)  =\left\{  x\in R^{+}:yx\leq y\text{ for all }y\in
R^{+}\right\}  . \label{2}%
\end{equation}
We claim that $\tau R$ is an $\ell$-ring. To do this, it suffices to prove
that the positive cone of $\tau R$ is closed under multiplication. Choose
$x,y\in R$ and $p,q\in\mathbb{Q}$ such that $0\leq x+p$ and $0\leq y+q$ in
$\tau R$. We derive, by $\left(  \upsilon_{7}\right)  $, that $p,q\geq0$. We
must show that $\left(  x+p\right)  \left(  y+q\right)  $ is positive in $\tau
R$. There is nothing to prove if $p=q=0$. So, assume that $p>0$ and $q=0$.
Hence,%
\[
\frac{1}{p}x^{-}\in\tau\left(  R^{+}\right)  \quad\text{and\quad}y\in R^{+}.
\]
But then%
\[
\frac{1}{p}x^{-}y\leq y
\]
(where we use (\ref{1})) and so%
\[
\left(  x+p\right)  y=x^{+}y-x^{-}y+py=x^{+}y+p\left(  y-\frac{1}{p}%
x^{-}y\right)  \in R^{+}.
\]
Accordingly, $\left(  x+p\right)  y$ is positive in $\tau R$, as desired. The
case where $p=0$ and $q>0$ can be obtained in the same way. Suppose now that
$p>0$ and $q>0$. We write%
\[
\left(  x+p\right)  \left(  y+q\right)  =xy+py+qx+pq.
\]
The proof will be complete once we show that%
\[
\frac{1}{pq}\left(  xy+py+qx\right)  ^{-}\in\tau\left(  R^{+}\right)  .
\]
Clearly,%
\[
xy+py+qx=u-v,
\]
where%
\[
u=x^{+}y^{+}+py^{+}+qx^{+}-x^{+}y^{-}-x^{-}y^{+}%
\]
and%
\[
v=py^{-}+qx^{-}-x^{-}y^{-}.
\]
Moreover, since $x+p$ and $y+q$ are positive in $\tau R$, we have%
\begin{equation}
\frac{1}{p}x^{-},\frac{1}{q}y^{-}\in\tau\left(  R^{+}\right)  . \label{b}%
\end{equation}
We obtain, by (\ref{1}) and (\ref{2}),%
\[
\frac{1}{p}x^{-}y^{+}\leq y^{+}\quad\text{and\quad}\frac{1}{q}x^{+}y^{-}\leq
x^{+}.
\]
It follows that%
\[
py^{+}-x^{-}y^{+}\in R^{+}\quad\text{and\quad}qx^{+}-x^{+}y^{-}\in R^{+}.
\]
We derive that $u\in R^{+}$. Analogously, (\ref{2}) and (\ref{b}) imply that%
\[
\frac{1}{q}x^{-}y^{-}\leq x^{-}%
\]
and so%
\[
qx^{-}-x^{-}y^{-}\in R^{+}.
\]
This shows that $v\in R^{+}$. Therefore,%
\begin{equation}
\frac{1}{pq}\left(  xy+py+qx\right)  ^{-}=\frac{1}{pq}\left(  u-v\right)
^{-}\leq\frac{1}{pq}v. \label{c}%
\end{equation}
On the other hand, pick $z\in R^{+}$ and observe that%
\[
\frac{1}{q}y^{-}z\leq z,
\]
where we use (\ref{1}) and (\ref{b}). So, again by (\ref{1}) and (\ref{b}),%
\[
\frac{1}{p}x^{-}\left(  z-\frac{1}{q}y^{-}z\right)  \leq z-\frac{1}{q}y^{-}z.
\]
Accordingly,%
\begin{align*}
\frac{1}{pq}vz  &  =\frac{1}{q}y^{-}z+\frac{1}{p}x^{-}z-\frac{1}{pq}x^{-}%
y^{-}z\\
&  =\frac{1}{q}y^{-}z+\frac{1}{p}x^{-}\left(  z-\frac{1}{q}y^{-}z\right) \\
&  \leq\frac{1}{q}y^{-}z+z-\frac{1}{q}y^{-}z=z.
\end{align*}
This together with (\ref{1}) gives that%
\[
\frac{1}{pq}v\in\tau\left(  R^{+}\right)  .
\]
Taking into consideration $\left(  \upsilon_{1}\right)  $ and (\ref{c}), we
derive that%
\[
\frac{1}{pq}\left(  xy+py+qx\right)  ^{-}\in\tau\left(  R^{+}\right)  ,
\]
which allows us to conclude.
\end{proof}

In light of Theorem \ref{main}, we can find simpler examples (in comparison
with Example \ref{Exp}) of an $\ell$-ring $R$ with a truncation $\tau$ such
that $\tau R$ fails to be an $\ell$-ring. Indeed, consider $\mathbb{R}$ with
its usual $\ell$-ring structure and $\tau\left(  x\right)  =2\wedge x$ as a
truncation. Clearly, $\tau\left(  \mathbb{R}^{+}\right)  \neq\left\{
x\in\mathbb{R}:x^{2}\leq x\right\}  $ and so $\tau\mathbb{R}$ is not an $\ell
$-ring. As an alternative consequence of Theorem \ref{main}, we shall prove
that there exists at most one truncation on an $\ell$-ring $R$ such that $\tau
R$ is an $\ell$-ring. To do this, we need the following lemma.

\begin{lemma}
Let $G$ be an $\ell$-group. Two truncations $\tau_{1}$ and $\tau_{2}$ on $G$
coincide if and only if $\tau_{1}\left(  G^{+}\right)  =\tau_{2}\left(
G^{+}\right)  $.
\end{lemma}

\begin{proof}
Only Sufficiency needs a proof. Suppose that $\tau_{1}\left(  G^{+}\right)
=\tau_{2}\left(  G^{+}\right)  $ and pick $x\in G^{+}$. We have to show that
$\tau_{1}\left(  x\right)  =\tau_{2}\left(  x\right)  $. On the one hand, we
have $\tau_{1}\left(  x\right)  \leq x$ and so $\tau_{2}\left(  \tau
_{1}\left(  x\right)  \right)  \leq\tau_{2}\left(  x\right)  $. On the other
hand, since $\tau_{1}\left(  x\right)  \in\tau_{1}\left(  G^{+}\right)
=\tau_{2}\left(  G^{+}\right)  $, we get $\tau_{2}\left(  \tau_{1}\left(
x\right)  \right)  =\tau_{1}\left(  x\right)  $. It follows that $\tau
_{1}\left(  x\right)  \leq\tau_{2}\left(  x\right)  $. Similarly, $\tau
_{2}\left(  x\right)  \leq\tau_{1}\left(  x\right)  $ and thus $\tau_{1}%
=\tau_{2}$, which is the desired equality.
\end{proof}

The next result turns out to be useful for a later purpose.

\begin{corollary}
\label{unique}Let $R$ be $\ell$-ring and $\tau_{1},\tau_{2}$ be two
truncations on $R$ such that both Alexandroff unitizations $\tau_{1}R$ and
$\tau_{2}R$ are $\ell$-rings. Then $\tau_{1}=\tau_{2}$.
\end{corollary}

\begin{proof}
By Theorem \ref{main}, we have%
\[
\tau_{1}\left(  G^{+}\right)  =\left\{  x\in G^{+}:x^{2}\leq x\right\}
=\tau_{2}\left(  G^{+}\right)  .
\]
The rest follows straightforwardly from the previous lemma.
\end{proof}

We are indebted to the referee for pointing out to us the following
interesting remark. Pick $c>0$ in $\mathbb{Q}$ and let $G$ be an $\ell$-group
with a truncation $\tau$. The map $\left(  x,p\right)  \rightarrow\left(
x,cp\right)  $ is a group automorphism of $G\oplus\mathbb{Q}$. Since the usual
unitization $\tau G$ of $G$ is known to be an $\ell$-group, we derive quickly
that the union%
\[
G^{+}\cup\left\{  x+p\in G\oplus\mathbb{Q}:p>0\text{ and }\frac{c}{p}x^{-}%
\in\tau\left(  G^{+}\right)  \right\}
\]
forms a positive cone on $G\oplus\mathbb{Q}$ under which it becomes an $\ell
$-group, denoted by $\left(  c\tau\right)  G$. Slight modifications of the
proof of Theorem \ref{main} yield, for an $\ell$-ring $R$ with a truncation
$\tau$, that $\left(  c\tau\right)  R$ is an $\ell$-ring if and only if $R$ is
a reduced $f$-ring with $\tau\left(  R^{+}\right)  =\left\{  x\in R:x^{2}\leq
cx\right\}  $. This could extend the applicability of our original unitization construction.

\section{The Archimedean case}

The aim of this section is to investigate the Alexandroff unitization of an
Archimedean $\ell$-ring with truncation. To begin with, we recall that an
$\ell$-group $G$ is said to be \textsl{Archimedean} if%
\[
\inf\left\{  \frac{1}{n}x:n=1,2,...\right\}  =0,\text{\quad for all }x\in
G^{+}.
\]
Clearly, any $\ell$-subgroup of an Archimedean $\ell$-group is again
Archimedean. In view of Axiom $\left(  \tau_{3}\right)  $, one might think
that any $\ell$-group with a truncation is automatically Archimedean. The next
example shows that this is not true.

\begin{example}
The Euclidean plane $G=\mathbb{R}^{2}$ is a totally-ordered $\ell$-group with
respect to coordinatewise addition and lexicographic ordering. It is readily
checked that the function $\tau:G^{+}\rightarrow G^{+}$ defined by%
\[
\tau\left(  \left(  r,s\right)  \right)  =\left(  0,1\right)  \wedge\left(
r,s\right)  ,\text{\quad for all }\left(  r,s\right)  \in G^{+}%
\]
is a truncation on $G$. But $G$ is not Archimedean.
\end{example}

The following Transfer's Type Theorem may well not have been quite on the
agenda, but we think that it could has some interest.

\begin{theorem}
Let $G$ be a $\ell$-group with a truncation $\tau$. Then $G$ is Archimedean if
and only if $\tau G$ is Archimedean.
\end{theorem}

\begin{proof}
Since $G$ is an $\ell$-subgroup of $\tau G$, the `if' part is obvious. To
prove the `only if' part, assume that $G$ is Archimedean. We claim that $\tau
G$ is Archimedean as well. To this end, pick $x,y\in G$ and $p,q\in\mathbb{Q}$
such that the inequality%
\[
0\leq n(x+p)\leq y+q
\]
holds in $\tau G$ for all $n\in\{1,2,...\}$. Using $\left(  \upsilon
_{7}\right)  $, we get $p\geq0$ and $q\geq0$. If $q=0$ then
\[
0\leq n(x+p)\leq y,\text{\quad for all }n\in\{1,2,...\}.
\]
This together with $\left(  \upsilon_{6}\right)  $ yields that $p=0$.
Therefore,
\[
0\leq nx\leq y,\text{\quad for all }n\in\{1,2,...\}.
\]
Since $G$ is Archimedean, we derive that $x=0$. Now, suppose that $q>0$.
Hence,
\[
0\leq y-nx+q-np,\text{\quad for all }n\in\{1,2,...\}.
\]
By $\left(  \upsilon_{2}\right)  $, we obtain%
\[
0\leq np\leq q,\text{\quad for all }n\in\left\{  1,2,...\right\}  ,
\]
so $p=0$. Accordingly,%
\[
0\leq y-nx+q,\text{\quad for all }n\in\{1,2,...\},
\]
from which it follows that%
\[
0\leq y-nmx+nq,\text{\quad for all }n,m\in\{1,2,...\}.
\]
Thus,%
\[
n\left(  mx-q\right)  \leq y,\text{\quad for all }n,m\in\{1,2,..\}.
\]
Consequently,%
\[
n\left(  mx-q\right)  ^{+}\leq y^{+},\text{\quad for all }n,m\in\{1,2,...\}.
\]
Using $(\upsilon_{3})$, we derive that
\[
0\leq n\left(  mx-q\tau\left(  \dfrac{m}{q}x\right)  \right)  \leq y^{+}\in
G,\text{\quad for all }n,m\in\{1,2,...\}.
\]
But then%
\[
\tau\left(  \dfrac{m}{q}x\right)  =\dfrac{m}{q}x,\text{\quad for all }%
m\in\{1,2,...\}
\]
because $G$ is Archimedean. This implies that $x=0$ (where we use $(\tau_{3}%
)$) and ends the proof of the theorem.
\end{proof}

Let's come back to the main problem of this section, namely, studying
Archimedean $f$-rings with truncation. We first have to recall some of the
relevant notions. Let $G$ be an Archimedean $\ell$-group. In \cite{CD1971},
Conrad and Diem call a group endomorphism $\psi$ on $G$ a $p$%
-\textsl{endomorphism} on $G$ if%
\[
x,y\in G\text{ and }x\wedge y=0\text{\quad imply\quad}\left(  \psi x\right)
\wedge y=0.
\]
Observe that any $p$-endomorphism $\psi$ on $G$ is \textsl{positive}, i.e.,%
\[
\psi x\in G^{+},\text{\quad for all }x\in G^{+}.
\]
A group endomorphism $\pi$ on $G$ is called \textsl{orthomorphism} on $G$ if
$\pi=\varphi-\psi$ for some $p$-endomorphisms $\varphi$ and $\psi$ on $G$. The
set of all orthomorphisms on $G$ is denoted by $\mathrm{Orth}\left(  G\right)
$. It is well known that $\mathrm{Orth}\left(  G\right)  $ is a Archimedean
$\ell$-group with respect to the pointwise addition and ordering. Moreover,
the composition operation makes $\mathrm{Orth}\left(  G\right)  $ into an
Archimedean $f$-ring with identity $\mathrm{id}_{G}$, where%
\[
\mathrm{id}_{G}x=x,\text{\quad for all }x\in G.
\]
For the basic properties of orthomorphisms on $\ell$-groups, we refer to the
books \cite{BKW1977,D1995}.

Now, let $R$ be an Archimedean $f$-ring. Notice in passing that, being
Archimedean, the $f$-ring $R$ is commutative (see, e.g., Th\'{e}or\`{e}me
12.3.2 in \cite{BKW1977}). For every $x\in R$, we define a map $\pi
_{x}:R\rightarrow R$ by putting%
\[
\pi_{x}y=xy,\text{\quad for all }y\in R.
\]
Obviously,%
\[
\pi_{x}\in\mathrm{Orth}\left(  R\right)  \text{ for all }x\in R.
\]
Hence, we may introduce a map $J:R\rightarrow\mathrm{Orth}\left(  R\right)  $
by putting%
\[
Jx=\pi_{x},\text{\quad for all }x\in R.
\]
Actually, $J$ is a lattice and ring homomorphism. The range of $J$ will be
denoted by $J\left(  R\right)  $, that is,%
\[
J\left(  R\right)  =\left\{  Jx:x\in R\right\}  .
\]
Hence, $J\left(  R\right)  $ is an $f$-subring of $\mathrm{Orth}\left(
R\right)  $. Furthermore, $J$ is injective (respectively, bijective) if and
only if $R$ is reduced (respectively, has an identity). In particular, if $R$
is reduced (respectively, has an identity) then $R$ and $J\left(  R\right)  $
(respectively, $\mathrm{Orth}\left(  R\right)  $) are isomorphic as $f$-rings.
From now on, we shall identify any Archimedean reduced $f$-ring $R$ with
$J\left(  R\right)  $ and so $R$ will be seen as an $f$-subring of
$\mathrm{Orth}\left(  R\right)  $. These observations will be used below
without further mention. But if the reader wants to look at the proofs, he can
consult \cite[Th\'{e}or\`{e}me 12.3.8 and Corollaire 12.3.13]{BKW1977}.

The following lemma is a consequence of Theorem \ref{main}. It will be used to
prove the next result.

\begin{lemma}
\label{original}Let $R$ be a Archimedean $\ell$-ring with a weak unit $e$.
Then $R$ is an $f$-ring with identity $e$ if and only if%
\[
\left\{  x\in R^{+}:x^{2}\leq x\right\}  =\left\{  x\in R:0\leq x\leq
e\right\}  .
\]

\end{lemma}

\begin{proof}
Assume that $R$ is an $f$-ring with $e$ as identity. Since $R$ is Archimedean,
it is reduced (see Theorem 1.11 (ii) in \cite{BH1990}). If $x\in R$ and $0\leq
x\leq e$ then, obviously, $x^{2}\leq x$. Conversely, choose $x\in R^{+}$ with
$x^{2}\leq x$. Since $R$ is reduced (see Corollaire 12.3.9 in \cite{BKW1977}),
Lemma \ref{semi} shows that $xy\leq y$ for all $y\in R$ and so $0\leq x\leq
e$. We have therefore proved Necessity. Let's prove Sufficiency. Assume that%
\[
\left\{  x\in R^{+}:x^{2}\leq x\right\}  =\left\{  x\in R:0\leq x\leq
e\right\}  .
\]
We have to show that $R$ is an $f$-ring with $e$ as identity. According to
Corollary 1.10 in \cite{BH1990}, it suffices to prove that $e$ is an identity
in $R$. It is readily checked that, putting%
\[
\tau\left(  x\right)  =e\wedge x\text{,\quad for all }x\in R^{+},
\]
we define a truncation $\tau$ on $R$. In particular, $R$ has $e$ as a unit for
the truncation $\tau$. Moreover, it follows directly from the hypothesis that%
\[
\tau\left(  R^{+}\right)  =\left\{  x\in R^{+}:x^{2}\leq x\right\}  .
\]
Thus, using Theorem \ref{main}, we derive that $\tau R$ is an $f$-ring. Now,
from Theorem 2.4 in \cite{BHM2018} it follows that $R^{\perp}=\mathbb{Q}%
\left(  1-e\right)  $, where $R^{\bot}$ is the polar of $R$ in $\tau R$. So,
the equality%
\[
\left\vert 1-e\right\vert \wedge\left\vert x\right\vert =0
\]
holds in $\tau R$ for all $x\in R$. But then $\left\vert 1-e\right\vert
\left\vert x\right\vert =0$ for every $x\in R$ because $\tau R$ is an
$f$-ring, and it follows that%
\[
xe=ex=x\text{,\quad for all }x\in R.
\]
This concludes the proof of the lemma
\end{proof}

Following \cite{BBE2014} (see also \cite{HP1986}), we call the Archimedean
reduced $f$-ring $R$ a \textsl{Stone }$f$-\textsl{ring} if%
\[
\mathrm{id}_{R}\wedge x\in R,\text{\quad for all }x\in R^{+}.
\]
For instance, any unital Archimedean $f$-ring is a Stone $f$-ring. Another
example is the $f$-ring $C_{0}\left(  \mathbb{R}\right)  $ of all continuous
real-valued functions on $\mathbb{R}$ that vanish at infinity. Indeed, it is
easily verified that $\mathrm{Orth}\left(  C_{0}\left(  \mathbb{R}\right)
\right)  $ can be identified with the $f$-ring $C^{\ast}\left(  \mathbb{R}%
\right)  $ of all bounded continuous real-valued functions on $\mathbb{R}$
(see \cite{Z1975}). This means in particular that any unital Archimedean
$f$-ring is a Stone $f$-ring. We call the \textsl{Stone function }on\textsl{
the Stone }$f$-ring $R$ the function $\tau:R^{+}\rightarrow R^{+}$ defined by%
\[
\tau\left(  x\right)  =\mathrm{id}_{R}\wedge x,\quad\text{for all }x\in
R^{+}.
\]
It is a very simple exercise to check that the Stone function on $R$ is a
truncation on $R$. The following lemma furnishes more information on Stone functions.

\begin{lemma}
\label{prop}The Stone function on a Stone $f$-ring $R$ is the unique
truncation $\tau$ on $R$ such that the Alexandroff unitization $\tau R$ of $R$
is an $\ell$-ring.
\end{lemma}

\begin{proof}
We already pointed out that the Stone function on the Stone $f$-ring $R$ is a
truncation on $R$. Furthermore, if $x\in\tau\left(  R^{+}\right)  $ then%
\[
x=\tau\left(  x\right)  =\mathrm{id}_{R}\wedge x
\]
and so $0\leq x\leq\mathrm{id}_{R}$. Multiplying these inequalities by $x$, we
get $x^{2}\leq x$. Conversely, if $x^{2}\leq x$ then $x\leq\mathrm{id}_{R}$
(where we use Lemma \ref{original} in the $f$-ring $\mathrm{Orth}\left(
R\right)  $). In summary, we have%
\[
\tau\left(  R^{+}\right)  =\left\{  x\in R^{+}:x^{2}\leq x\right\}  .
\]
This together with Theorem \ref{main} yields that $\tau R$ is an $\ell$-ring.
Now, uniqueness follows straightforwardly from Corollary \ref{unique} and the
proof is complete.
\end{proof}

In order to prove the central theorem of this section, we need speak a little
about truncation homomorphisms (see \cite{B2014,BE2017} for more information).
Let $G_{1}$ and $G_{2}$ be two $\ell$-groups with truncations $\tau_{1}$ and
$\tau_{2}$, respectively. An $\ell$-homomorphism $h:G_{1}\rightarrow G_{2}$ is
called a \emph{truncation homomorphism} if $h$ preserves truncation, i.e.,
$\tau_{2}\circ h=h\circ\tau_{1}$. Assume now that $G_{k}$ contains a unit
$e_{k}$ for $\tau_{k}$ ($k=1,2$). The lattice homomorphism $h:G_{1}\rightarrow
G_{2}$ is said to be \emph{unital} if $he_{1}=e_{2}$. We easily check that any
unital lattice homomorphism is a truncation homomorphism.

\begin{theorem}
Let $R$ be an Archimedean $\ell$-ring with a truncation $\tau$. Then $\tau R$
is an $\ell$-ring if and only if $R$ is a Stone $f$-ring and $\tau$ is the
Stone function.
\end{theorem}

\begin{proof}
Sufficiency was proved in Lemma \ref{prop}. Let's establish Necessity by
assuming that $\tau R$ is an $\ell$-ring.

First, suppose that $R$ contains a unit $e>0$ for the truncation $\tau$. In
particular, we have%
\[
\tau\left(  R^{+}\right)  =\left\{  x\in R:0\leq x\leq e\right\}  .
\]
In view of Theorem \ref{main}, we get the equality%
\[
\left\{  x\in R^{+}:x^{2}\leq x\right\}  =\left\{  x\in R:0\leq x\leq
e\right\}  .
\]
Moreover, $e$ is a weak unit in $R$ and thus, by Lemma \ref{original}, we
deduce that $R$ is an Archimedean $f$-ring with identity $e$. But then $R$ can
be identified with $\mathrm{Orth}\left(  R\right)  $ and so $R$ is a Stone
$f$-ring. Uniqueness in Lemma \ref{prop} shows that $\tau$ is the Stone
function on $R$, indeed.

Now, assume that $R$ contains no unit for the truncation $\tau$. In view of
Uniqueness in Lemma \ref{prop}, the proof will be complete once we show that
$R$ is a Stone $f$-ring. By Proposition \ref{reduced}, $R$ is a reduced
$f$-ring. Furthermore, $\mathrm{Orth}\left(  R\right)  $ can directly be
equipped with a truncation $\zeta$ by putting%
\[
\zeta\left(  \pi\right)  =\pi\wedge\mathrm{id}_{R}\text{,\quad for all }\pi
\in\mathrm{Orth}\left(  R\right)  ^{+}.
\]
We claim that the embedding $J:R\rightarrow\mathrm{Orth}\left(  R\right)  $ is
a truncation homomorphism. To this end, recall that multiplications by
positive elements in an $f$-ring are $\ell$-homomorphisms. Now, choose $x\in
R^{+}$ and observe that if $y\in R^{+}$ then%
\[
J\left(  \tau\left(  x\right)  \right)  \left(  y\right)  =\tau\left(
x\right)  y=\left(  x\wedge1\right)  y=xy\wedge y
\]
(where we use $\left(  \upsilon_{5}\right)  $. Hence,%
\begin{align*}
\zeta\left(  Jx\right)  \left(  y\right)   &  =\left(  \mathrm{id}_{R}\wedge
Jx\right)  \left(  y\right)  =\left(  Jx\right)  \left(  y\right)  \wedge y\\
&  =xy\wedge y=J\left(  \tau\left(  x\right)  \right)  \left(  y\right)  .
\end{align*}
This means that $\zeta\circ J=J\circ\tau$ and so $J$ is a truncation
homomorphism, as desired. Thus, we may use Corollary 3.2 in \cite{BHM2018} to
infer that $J$ extends uniquely to a one-to-one unital lattice homomorphism
$J^{\tau}:\tau R\rightarrow\mathrm{Orth}\left(  R\right)  $. Choose $x\in
R^{+}$ and observe that%
\begin{align*}
\mathrm{id}_{R}\wedge Jx  &  =\mathrm{id}_{R}\wedge J^{\tau}x\\
&  =J^{\tau}1\wedge J^{\tau}x\\
&  =J^{\tau}\left(  1\wedge x\right)  =J^{\tau}\left(  \tau\left(  x\right)
\right) \\
&  =J\left(  \tau\left(  x\right)  \right)  =\tau\left(  x\right)  \in R.
\end{align*}
This shows that $R$ is a Stone $f$-ring, completing the proof of the theorem.
\end{proof}

\end{document}